\documentclass[leqno]{amsart}
\usepackage{amsfonts, amssymb, amsmath, amscd, lastpage, fancyheadings, array}

\newtheorem{thm}{Theorem}
\newtheorem{theorem}{Theorem}[section]

\newtheorem{lem}{Lemma}
\newtheorem{lemma}[theorem]{Lemma}

\newtheorem{coro}[theorem]{Corollary}

\newcommand{\al}{\alpha}

\newcommand{\om}{\omega}
\newcommand{\del}{\delta}

\newcommand{\lf}{\big{\lfloor}}
\newcommand{\rf}{\big{\rfloor}}
\newcommand{\Z}{\mbox{$\mathbb Z$}}

\newcommand{\R}{\mbox{$\mathbb R$}}     

\begin{document}
\title{Irreducibility of generalized Hermite-Laguerre Polynomials II}
\author{Shanta Laishram}
\email{shanta@iiserbhopal.ac.in, shantalaishram@gmail.com}
\address{Indian Institute of Science Education and Research, \\
ITI Campus (Gas Rahat) Building, Govindpura, Bhopal - 462 023, India}
\author[T. N. Shorey]{T. N. Shorey}
\email{shorey@math.tifr.res.in}
\address{School of Mathematics\\
Tata Institute of Fundamental Research\\
Homi Bhabha Road, Mumbai 400005, India}
\thanks{Submitted on \today}
\begin{abstract}
In this paper, we show that for each $n\geq 1$, the generalised Hermite-Laguerre Polynomials
$G_{\frac{1}{4}}$ and $G_{\frac{3}{4}}$ are either irreducible or linear
polynomial times an irreducible polynomial of degree $n-1$.
\end{abstract}

\maketitle

\section{Introduction}

Let $n$ and $1\le \al<d$ be positive integers with gcd$(\al, d)=1$. Let
$q=\frac{\al}{d}$ and let
\begin{align*}
(\al)_j=\al(\al+d)\cdots (\al +(j-1)d)
\end{align*}
for non negative integer $j$. We define
\begin{align*}
F(x):=F_{q}(x)=a_n\frac{d^nx^n}{(\al)_n}+
a_{n-1}\frac{d^{n-1}x^{n-1}}{(\al)_{n-1}}+\cdots +
a_1\frac{dx}{(\al)_{1}}+a_0
\end{align*}
where $a_0, a_1, \cdots a_n\in \Z$ and $P(|a_0a_n|)\leq 2$. Here $P(\nu)$ is the
maximum prime divisor for $|\nu|>1$ and $P(1)=P(-1)=1$. We put
\begin{align*}
G(x):=G_q(x)=&(\al)_{n}F_q(\frac{x}{d})\\
=&a_nx^n+a_{n-1}(\al +(n-1)d)x^{n-1}+\cdots + \\
&a_1\left(\prod^{n-1}_{i=1}(\al +id)\right)x+a_0
\left(\prod^{n-1}_{i=0}(\al +id)\right).
\end{align*}
Schur \cite{schur} proved that $G_{\frac{1}{2}}$ with $|a_0|=|a_n|=1$ is irreducible.
Laishram and Shorey \cite{st1/3} showed that $G_{\frac{1}{3}}$ and $G_{\frac{2}{3}}$ are
either irreducible or linear polynomial times an irreducible polynomial of degree $n-1$
whenever $|a_0|=|a_n|=1$. For an account of earlier results, we refer to \cite{stirred}
and \cite{fifile}. We prove

\begin{thm}\label{1/4}
For each $n$, the polynomials $G_{\frac{1}{4}}$ and $G_{\frac{3}{4}}$ 
are either irreducible or linear polynomial times an irreducible polynomial of degree $n-1$.
\end{thm}

For Theorem \ref{1/4}, we prove the following lemma in Section 2.

\begin{lem}\label{mainupdate}
Let $1\le k\le \frac{n}{2}$. Suppose there is a prime $p$ satisfying
\begin{align*}
p>d, p\ge \min(2k, d(d-1))
\end{align*}
and
\begin{align}\label{condak}
p|\prod^{k}_{j=1}(\al +(n-j)d), \ \ p\nmid \prod^{k}_{j=1}(\al +(j-1)d).
\end{align}
Then $G(x)$ has no factor of degree $k$.
\end{lem}

We compare Lemma \ref{mainupdate} with \cite[Lemma 10.1]{stirred}. The assumption on
$p$ in \cite[Lemma 10.1]{stirred} has been relaxed.
For any integer $\nu >1$, we denote by $\om(\nu)$ the number of distinct prime factors of $\nu$
and $\om(1)=0$. In Section $3$, we give an upper bound for $m$ when
$\om(\prod^{k-1}_{i=0}(m+id))\le t$ for some $t$. In Section $4$, we give preliminaries
for the proof of Theorem \ref{1/4}. In Section $5$, we complete the proof.

\section{Proof of Lemma \ref{mainupdate}}

Let
\begin{align*}
\Delta_j=\al (\al +d )\cdots (\al +(j-1)d ).
\end{align*}
For each $1\le l<d$ and gcd$(l, d)=1$, we observe that $q|\Delta_k$ for all
primes $q\equiv l^{-1}\al($mod $d)$ and $q\le \frac{kd}{l}$.
Since $p>\al $ and $p\nmid \Delta_k$, we have $p>\frac{kd}{d-1}$.
Let $j_0$ be the minimum $j$ such that $p|(\al +(j-1)d )$ and we write $\al +(j_0-1)d=pl_0$.
Then $j_0>k$ since $p\nmid \Delta_k$ and we observe that $1\le l_0<d$ by the minimality of
$j_0$. As in the proof of \cite[Corollary 2.1]{stirred}, it suffices to show that
\begin{align*}
\phi_j=\frac{{\rm ord}_p(\Delta_j)}{j}<\frac{1}{k} \ \ {\rm for} \ 1\le j\le n.
\end{align*}
We may restrict to those $j$ such that $\al +(j-1)d =pl$ for some $l$. Then
$(j-j_0)d =p(l-l_0)$ implying $d |(l-l_0)$. Writing $l=l_0+sd$, we get
$j=j_0+ps$. Note that if $p|(\al +(i-1)d )$, then $\al +(i-1)d =p(l_0+r d)$
for some $r\ge 0$. Hence we have
\begin{align}\label{l_or}
{\rm ord}_p(\Delta _j)&={\rm ord}_p((pl_0)(p(l_0+d))\cdots (p(l_0+sd ))=
s+1+{\rm ord}_p(l_0(l_0+d)\cdots (l_0+sd))
\end{align}
for some integer $s\ge 0$. Further we may suppose that $s>0$ otherwise the assertion
follows since $p>d>l_0$. Let $r_0$ be such that ord$_p(l_0+r_0d )$ is maximal. We consider two cases.

\noindent
{\bf Case I:} Assume that $s<p$. Then $p$ divides at most one term of
$\{l_0+id : 0\le i\le s\}$ and we obtain from \eqref{l_or} and
$l_0+sd<(s+1)d<p^2$ that $\phi_j\le \frac{s+2}{j_o+ps}$. Thus
$\phi_j<\frac{1}{k}$ if $s(p-k)\ge k$ since $j_0-k+s(p-k)-k\ge 1+s(p-k)-k$.
If $p\ge 2k$, then $s(p-k)\ge k$. Thus we may suppose that $p<2k$. Then
$p\ge d(d-1)$. Since $p>\frac{kd}{d-1}$, we obtain $s(p-k)\ge k$ if $s\ge d-1$.
We may suppose $s\le d-2$. Then $l_0+sd\le d-1+(d-2)d<p$ and therefore
$\phi_j=\frac{s+1}{j_0+ps}\le \frac{s+1}{k+1+(k+1)s}<\frac{1}{k}$.

\noindent
{\bf Case II:} Let $s\ge p$.  Then
\begin{align*}
{\rm ord}_p(\Delta _j)\le s+1+{\rm ord}_p(l_0+r_0d)+ {\rm ord}_p(s!)\le s+1+
\frac{\log (l_0+sd)}{\log p}+\frac{s}{p-1}.
\end{align*}
We have $p\ge d+1$. This with $l_0\le d-1<p\le s$ imply
$\log (l_0+sd)\le \log s(d+1)=\log s +\log (d+1)\le \log s +\log p$. Hence
\begin{align*}
{\rm ord}_p(\Delta _j)\le s+1+\frac{s}{p-1}+\frac{\log s}{\log p}+1.
\end{align*}
Since $\frac{j}{k}=\frac{j_0+ps}{k}>1+\frac{p}{k}s$, it is enough to show that
\begin{align*}
\frac{p}{k}\ge 1+\frac{1}{p-1}+\frac{1}{s}+\frac{\log s}{s\log p}
\end{align*}
Since $s\ge p$, the right hand side of the above inequality is at most $1+\frac{1}{p-1}+
\frac{2}{p}$ and therefore it suffices to show
\begin{align}\label{iterm2}
1+\frac{1}{p-1}+\frac{2}{p}\le \frac{p}{k}.
\end{align}
Let $p\ge 2k$. Then $p\ge 2k+1\ge k+2$ and the left hand side of \eqref{iterm2} is at most
\begin{align*}
1+\frac{1}{2k}+\frac{2}{2k+1}\le 1+\frac{2}{k}=\frac{k+2}{k}\le \frac{p}{k}.
\end{align*}
Thus we may assume that $p<2k$. Then $p>d(d-1)$ since $p\nmid d$. Further $d\ge 3$ since
$p\ge \frac{kd}{d-1}$. Therefore the left hand side of \eqref{iterm2} is at most
\begin{align*}
1+\frac{3}{d(d-1)}\le 1+\frac{1}{d-1}=\frac{d}{d-1}\le \frac{p}{k}.
\end{align*}
Hence the proof.
$\hfill \Box$

\section{An upper bound for $m$ when $\om(\Delta (m, d, k))\le t$}\label{omdbd}

Let $m$ and $k$ be positive integers with $m>kd$ and gcd$(m,d) =1$. We write
\begin{align*}
\Delta (m, d, k)=m(m+d) \cdots (m+(k-1)d).
\end{align*}
Assume that
\begin{align}\label{piDk}
\om(\Delta (m, d, k))\le t.
\end{align}
for some integer $t$. For every prime $p$ dividing $\Delta $, we delete a term
$m+i_pd$ such that ord$_p(m+i_pd)$ is maximal.
Then we have a set $T$ of terms in $\Delta(m, k)$ with
\begin{align*}
|T|=k-t:=t_0.
\end{align*}
We arrange the elements of $T$ as $m+i_1d<m+i_2d<\cdots <m+i_{t_0}d$. Let
\begin{align}\label{P0}
{\frak P}:=\displaystyle{\prod^{t_0}_{\nu =1}} (m+i_{\nu}d) \geq m^{t_0}.
\end{align}

Now we deduce an upper bound for ${\frak P}$. For a prime $p$, let $r$ be
the highest power of $p$ such that $p^r\le k-1$. Let $w_l=\#\{m+id: p^l|(m+i), m+i\in T\}$
for $1\le l\le r$. By Sylvester and Erd\H{o}s argument, we have
$w_l\le [\frac{i_0}{p^l}]+[\frac{k-1-i_o}{p^l}]\le [\frac{k-1}{p^l}]$.
Let $h_p>0$ be such that
$[\frac{k-1}{p^{h_p+1}}]\le t_0<[\frac{k-1}{p^{h_p}}]$. Then
$|\{m+id\in T: {\rm ord}_p(n+id)\le h_p\}|\le t_0-w_{h_p+1}$.
Hence
\begin{align*}
{\rm ord}_p({\frak P})&\le rw_r +
\sum^{r-1}_{u=h_p+1}u(w_u-w_{u+1})+h_p(t_0-w_{h_p+1})\\
&=w_r+w_{r-1}+\cdots +w_{h_p+1}+h_pt_0\\
&\le \sum^{r}_{u=1}\lf\frac{k-1}{p^u}\rf +h_pt_0-
\sum^{h_p}_{u=1}\lf \frac{k-1}{p^u}\rf ={\rm ord}_p((k-1)!)+
h_pt_0-\sum^{h_p}_{u=1}\lf \frac{k-1}{p^u}\rf .
\end{align*}
It is also easy to see that ord$_p({\frak P})\le $ord$_p(k-1)!)$ if
$p\nmid d$ and ord$_p({\frak P})=0$ if $p|d$.
Therefore
\begin{align*}
m^{t_0}\le {\mathfrak P} \leq (k-1)! \prod_{p\leq k}p^{L_0(p)}
\end{align*}
where
\begin{align*}
L_0(p)=\begin{cases} \min (0, h_pt_0-\sum^{h_p}_{u=1}\lf \frac{k-1}{p^u}\rf ) & {\rm if} \ p\nmid d\\
-{\rm ord}_p((k-1)!) & {\rm if} \ p| d.
\end{cases}
\end{align*}
Observe that
\begin{align}\label{E0}
m^{t_0}\le (k-1)!\prod_{p|d}p^{-{\rm ord}_{p}((k-1)!)}.
\end{align}
We also note that $L_0(p)\leq 0$ for any prime $p$. Hence for any $l\ge 1$, we
have from \eqref{P0} that
\begin{align}\label{Lbd}
m\le \left({\mathfrak P}\right)^{\frac{1}{t_0}} \le
\left( (k-1)! \prod_{p\le p_l}p^{L_0(p)}\right)^{\frac{1}{t_0}}=:L(k, l).
\end{align}

\section{Preliminaries for Theorems \ref{1/4}}

Let $m$ and $k$ be positive integers with $m>kd$ and gcd$(m, d) =1$. We write
\begin{align*}
\Delta (m, d, k)=m(m+d) \cdots (m+(k-1)d).
\end{align*}
For positive integers $\nu, \mu$ and $1\le l<\mu$ with gcd$(l, \mu)=1$, we write
\begin{align*}
\pi(\nu, \mu, l)=&\sum_{\underset{p\equiv l({\rm mod} \ \mu)}{p\le \nu}} 1, \
\pi(\nu)=\pi(\nu, 1, 1)\\
\theta(\nu, \mu, l)=&\sum_{\underset{p\equiv l({\rm mod} \ \mu)}{p\le \nu}}
\log p.
\end{align*}
Let $p_{i, \mu, l}$ denote the $i$th prime congruent to $l$ modulo $\mu$. Let
$\del_{\mu}(i, l)=p_{i+1, \mu, l}-p_{i, \mu, l}$ and
$W_\mu (i, l)=(p_{i, \mu, l}, p_{i+1, \mu, l})$.
We recall some well-known estimates from prime number theory.

\begin{lemma}\label{pix} Let $k\in \Z$ and $\nu\in \R$ be positive. We have
\begin{itemize}
\item[$(i)$] $\pi (\nu) \le \left(1+\frac{1.2762}{\log \nu}\right)$ for $\nu >1$
\item[$(ii)$] {\rm ord}$_p(k-1)! \geq \frac{k-p}{p-1}-\frac{\log (k-1)}{\log p}$ for $k\ge 2$.
\item[$(iii)$] $\sqrt{2\pi k}~e^{-k}k^{k}e^{\frac{1}{12k+1}} <k!<
\sqrt{2\pi k}~e^{-k}k^{k} e^{\frac{1}{12k}}$.
\end{itemize}
\end{lemma}

The estimates $(i)$ is due to Dusart(\cite{dus1}. The estimate $(iii)$ is due to Robbins \cite[Theorem 6]{robb}.
For a proof of $(ii)$, see \cite[Lemma 2(i)]{shanta2}.
\qed

The following lemma is due to Ramar\'e and Rumely \cite[Theorems 1, 2]{Rama}.

\begin{lemma}\label{ramar} Let $d=4$ and $l\in \{1, 3\}$. For $\nu_0\le 10^{10}$, we have
\begin{align}\label{lthe}
\theta(\nu, d, l)\ge \begin{cases}
\frac{\nu}{2}(1-0.002238) \ & {\rm for} \ \nu\ge 10^{10}\\
\frac{\nu}{2}\left(1-\frac{2\times 1.798158}
{\sqrt{\nu_0}}\right) \ &{\rm for} \ 10^{10}>\nu\ge \nu_0
\end{cases}
\end{align}
and
\begin{align}\label{uthe}
\theta(\nu, d, l)\le \begin{cases}
\frac{\nu}{2}(1+0.002238) \ &{\rm for} \ \nu\ge 10^{10}\\
\frac{\nu}{2}\left(1+\frac{2\times 1.798158}
{\sqrt{\nu_0}}\right) \ &{\rm for} \ 10^{10}>\nu\ge \nu_0.
\end{cases}
\end{align}
\end{lemma}

We derive from Lemmas \ref{pix} and \ref{ramar} the following result.

\begin{coro}\label{n<<}
Let $10^6<m\le 138\times 4k$. Then $P(\Delta(m, 4, k))\ge m$.
\end{coro}

\begin{proof}
Let $d=4$ and $10^6\le m\le 138\times dk$. Let $l\in \{1, 3\}$ and assume $m\equiv l($mod $d)$.
We observe that $P(\Delta(m, d, k)\ge m$ holds if
\begin{align*}
\theta(m+d(k-1), d, l)-\theta(m-d, d, l)=\sum_{\underset{p\equiv l(d)}
{m<p\le m+(k-1)d}}\log p >0.
\end{align*}
From Lemmas \ref{pix} and \ref{ramar}, we have
\begin{align*}
\frac{\theta(m-d, d, l)}{\frac{m-d}{\phi(d)}}<1+\frac{2\times 1.798158}{\sqrt{10^6}}
\end{align*}
and
\begin{align*}
\frac{\theta(m+(k-1)d, d, l)}{\frac{m-d+dk}{\phi(d)}}>
1-\frac{2\times 1.798158}{\sqrt{10^6}}
\end{align*}
Thus $P(\Delta(m, d, k)\ge m$ holds if
\begin{align*}
(1-\frac{2\times 1.798158}{10^3})dk>\frac{4\times 1.798158}{10^3}(m-d)
\end{align*}
which is true since
\begin{align*}
\frac{m}{dk}\le 138<\frac{10^3}{4\times 1.798158}-\frac{1}{2}.
\end{align*}
Hence the assertion.
\end{proof}

The following lemma is a computational result.

\begin{lemma}\label{diff}
Let $l\in \{1, 3\}$. Then $\del_4(i, l)\le 24, 32, 60, 200$
according as $p_{i, 4, l}\le 120, 250, 2400, 10^6$, respectively.
\end{lemma}

As a consequence, we obtain

\begin{coro}\label{<250}
Let $d=4$, $k\ge 6$ and $m$ be such that $m\le 120, 250, 2400, 10^6$
when $6\le k<8$, $8\le k<15$, $15\le k<50$ and $k\ge 50$ respectively.
Then $P(\Delta(m, d, k))\ge m$.
\end{coro}

\begin{proof}
We may assume that $p_{i, d, l}<m<m+(k-1)d<p_{i+1, d, l}$ for some $i$ otherwise
the assertion follows. Thus $p_{i+1, d, l}\geq d+m+(k-1)d$ and $p_{i, d, l}\leq m-d$.
Therefore $\del_d(i, l)=p_{i+1, d, l}-p_{i, d, l}\ge d+m+(k-1)d-(m-d)=d(k+1)>dk$.
Now the assertion follows from Lemma \ref{diff}.
\end{proof}

\section{Proof of Theorem \ref{1/4}}\label{<p250}

Let $2\le k\le \frac{n}{2}$ and assume that $G(x)$ has a factor of degree $k$.
We take $m=\al+4(n-k)$. Since $n\ge 2k$, we have $m>4k$. We may assume that
$P(\Delta(m, 4, k))\le 4k$ otherwise the assertion follows from Lemma \ref{mainupdate}
since $\al +4(k-1)<4k$. Thus $P(\Delta(m, 4, k))\le 4k<m$.

Let $k\le 6$. Then $P(\Delta(m, 4, k))\le 4k\le 23$ implying $P(m(m+4))\le 24$. Then
$m+4=N$ where $N$ is given by \cite[Table IIA]{lehmer} for $p\le 23$.
For each such $N$ and for each $2\le k\le 6$, we first
restrict to those $m=N-4>4k$ such that $P(\Delta(m, 4, k))\le 4k$. They
are given by $k=2$, $m\in \{21, 45\}$. Here $P(m(m+4))=7$ and since
$m\equiv 1$ modulo $4$, the assertion follows by taking $p=7$ in Lemma
\ref{mainupdate}.

Therefore $k\ge 7$. Let $\om_1(k):=\underset{\al \in \{1, 3\}}{\max}\omega(\Delta(\al, 4, k))$.
If $\om(\Delta(m, 4, k))>\om_1$, then there is a prime $p$ satisfying \eqref{condak} implying
$p>k\ge 7$. Observe that $11|\Delta(3, 4, k)$ and $11|\Delta(1, 4, k)$ for $k\ge 9$. For
$k\in \{7, 8\}$, if $\om(\Delta(m, 4, k))>\om_1$, then there are two primes $p>k$ dividing
$\Delta(m, 4, k)$ but $p\nmid \Delta(1, 4, k)$ and hence there is a prime $p>11$ satisfying
\eqref{condak}. Therefore
by Lemma \ref{mainupdate}, we may assume that $\om(\Delta(m, 4, k))\le \om_1$. Taking $t=\om_1$,
we obtain from \eqref{Lbd} with $p_l=7$ that $m\le 104, 245, 2353$ according as $k\le 10, 20, 400$,
respectively. This is not possible by Corollary \ref{<250}.

Hence $k>400$ and further $m>10^6$ by Corollary \ref{<250}. By Corollary \ref{n<<},
we may further suppose that $m\geq  v_0\cdot 4k$ where $v_0:=138$.
Since $P(\Delta(m, d, k))\le 4k$, we have $\om(\Delta(m, d, k))\le \pi(4k)-1$.
Taking $t=\pi(4k)-1$ in \eqref{piDk}, we obtain from \eqref{E0} that
\begin{align*}
(v_0\cdot 4k)^{k-\pi(4k)+1}\le (k-1)!2^{-{\rm ord}_2((k-1)!)}=\frac{k!}{k}2^{-{\rm ord}_2((k-1)!)}.
\end{align*}
By using estimates of ord$_p(k-1)!)$ and $k!$ from Lemma \ref{pix}, we obtain
\begin{align*}
(v_0\cdot 4k)^{k-\pi(4k)}<\frac{1}{k(v_0\cdot 4k)}(\frac{k}{e})^k
\left((2\pi k)^{\frac{1}{2}}{\rm exp}(\frac{1}{12k})\right)(2^{-k+2}(k-1)
\end{align*}
or
\begin{align*}
(v_0\cdot 4\cdot e\cdot 2)^k<(v_0\cdot 4k)^{\pi(4k)}
\frac{\left((2\pi)^{\frac{1}{2}}{\rm exp}(\frac{1}{12k})\right)}{v_0\cdot \sqrt{k}}
<(v_0\cdot 4k)^{\pi(4k)}
\end{align*}
since $k>400$. By using estimates of $\pi(4k)$ from Lemma \ref{pix}, we get
\begin{align*}
\log (v_0\cdot 8\cdot e)<\frac{4\log (v_0\cdot 4k)}{\log (4k)}
\left(1+\frac{1.2762}{\log (4k)}\right).
\end{align*}
The right hand side of the above expression is a decreasing function of $k$
and the inequality does not hold at $k=401$. This is a contradiction.
\qed

\section*{Acknowledgments}

A part of this work was done when the second author was visiting Max-Planck Institute for Mathematics
in Bonn during Feburary-April, 2009 and he would like to thank MPIM for the invitation and the
hospitality. We would also like to thank the referee for his comments on an earlier draft of the paper.

\end{document}